\documentclass{article}
\usepackage[latin1]{inputenc}
\usepackage[T1]{fontenc}
\usepackage{amsmath,amssymb}
\usepackage{amsfonts,amsthm}
\usepackage{geometry}
\usepackage[hypertexnames=false]{hyperref}
\usepackage{hypcap}
\usepackage{graphicx}
\usepackage{caption}
\usepackage{subcaption}

\newcommand{\C}{\mathbb{C}}
\newcommand{\D}{\mathbb{D}}
\newcommand{\N}{\mathbb{N}}
\newcommand{\Q}{\mathbb{Q}}
\newcommand{\Z}{\mathbb{Z}}

\newcommand{\Lc}{\mathcal{L}}

\newcommand{\from}{\colon}

\DeclareMathOperator{\clos}{clos}
\DeclareMathOperator{\spec}{sp}
\DeclareMathOperator{\diag}{diag}

\renewcommand{\phi}{\varphi}

\providecommand{\abs}[1]{\left\lvert#1\right\rvert}
\providecommand{\set}[1]{\left\{ #1\right\}}

\newtheorem{thm}{Theorem}
\newtheorem{lem}[thm]{Lemma}

\begin{document}

\title{The Eigenvalues of Tridiagonal Sign Matrices are Dense in the Spectra of Periodic Tridiagonal Sign Operators}
\author{Raffael Hagger\footnote{raffael.hagger@tuhh.de}}
\maketitle

\begin{abstract}
Chandler-Wilde, Chonchaiya and Lindner conjectured that the set of eigenvalues of finite tridiagonal sign matrices (i.e. plus and minus ones on the first sub- and superdiagonal, zeroes everywhere else) is dense in the set of spectra of periodic tridiagonal sign operators on the usual Hilbert space of square summable bi-infinite sequences. We give a simple proof of this conjecture. As a consequence we get that the set of eigenvalues of tridiagonal sign matrices is dense in the unit disk. In fact, a recent paper further improves this result, showing that this set of eigenvalues is dense in an even larger set.
\end{abstract}



\section{Introduction}

Let $n,m \in \N$. For $k,l \in \set{\pm 1}^n$ we define the corresponding (finite) tridiagonal matrix
\[A^{k,l}_{fin} := \begin{pmatrix} 0 & l_1 & & \\ k_1 & \ddots & \ddots & \\ & \ddots & \ddots & l_n \\ & & k_n & 0 \end{pmatrix} \in \C^{(n+1) \times (n+1)}.\]
Similarly we define the $m$-periodic tridiagonal operator
\[A^{k,l}_{per} := \begin{pmatrix} \ddots & \ddots & & & & & & \\ \ddots & 0 & l_m & & & & & \\ & k_m & 0 & l_1 & & & &\\ & & k_1 & \ddots & \ddots & & & \\ & & & \ddots & \ddots & l_m & & \\ & & & & k_m & 0 & l_1 & \\ & & & & & k_1 & 0 & \ddots \\ & & & & & & \ddots & \ddots \end{pmatrix} \in \Lc(\ell^2(\Z))\]
for $k,l \in \set{\pm 1}^m$. In the following we will compare the spectra, denoted by $\spec(\cdot)$, of the finite matrices with the spectra of the periodic operators as follows. In accordance with the notation in \cite{ChaChoLi} and \cite{ChaChoLi2}, we define the sets $\sigma_\infty$ and $\pi_\infty$:
\[\sigma_\infty := \bigcup\limits_{n \in \N} \bigcup\limits_{k,l \in \set{\pm 1}^n} \spec(A^{k,l}_{fin}), \qquad \pi_\infty := \bigcup\limits_{m \in \N} \bigcup\limits_{k,l \in \set{\pm 1}^m} \spec(A^{k,l}_{per}).\]
These sets are of interest because they are closely related to the spectrum of the Feinberg-Zee Random Hopping Matrix, denoted by $\Sigma$, see e.g. \cite{ChaChoLi}-\cite{ChaDa}, \cite{FeZee}, \cite{Ha}-\cite{HoOrZee}. It was shown in \cite{ChaChoLi2} that
\[\sigma_\infty \subset \pi_\infty \subset \Sigma\]
holds. Note that these inclusions have to be proper because $\sigma_\infty$ is a countable set, $\pi_\infty$ is an uncountable null set (w.r.t. $2$-dimensional Lebesgue measure) and $\Sigma$ contains the closed unit disk (see below). The authors of \cite{ChaChoLi2} conjectured that $\sigma_\infty$ is actually dense in $\Sigma$. We make a step in this direction by proving that $\sigma_\infty$ is at least dense in $\pi_\infty$.

Moreover, it was shown in \cite{ChaDa} that $\pi_\infty$ is dense in the unit disk $\D := \set{z \in \C: \abs{z} \leq 1}$, i.e. $\clos(\pi_\infty) \cap \D = \D$. This result was improved to a larger set in \cite{Ha2}. Our new result here implies that $\sigma_\infty$ is dense in this particular set as well. On the other hand, using the numerical range, one can show that $\Sigma$ is contained in the square with vertices $\set{2,2i,-2,-2i}$ (\cite{ChaChoLi2}, \cite{Ha}). In \cite{Ha} this upper bound was further improved using second order numerical ranges. Furthermore, $\pi_\infty$ has an infinite number of symmetries that also carry over to $\sigma_\infty$ (see \cite{ChaChoLi2} for rotations and reflections, \cite{ChaDa} for square roots and \cite{Ha2} for higher roots). Combining all these results mentioned above, we have a pretty accurate description of the set $\sigma_\infty$. However, many questions still remain open. For example the seemingly fractal boundary is yet unexplained (cf. Figure 1, borrowed from \cite{ChaChoLi2}). Some comments on this particular issue can be found in \cite{Ha2}.

\begin{figure}
\begin{center}
\begin{subfigure}{0.4\textwidth}
\includegraphics[width = \textwidth]{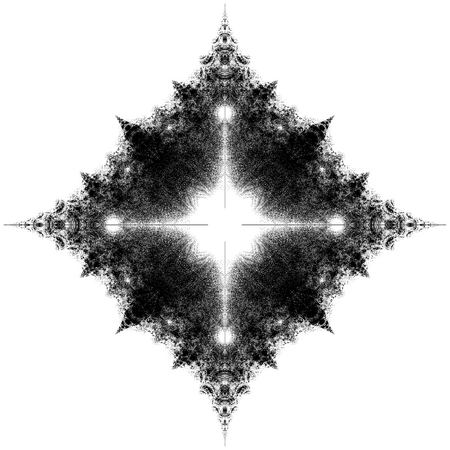}
\caption{\footnotesize eigenvalues of all $20 \times 20$ tridiagonal sign matrices}
\end{subfigure}
\qquad
\begin{subfigure}{0.4\textwidth}
\includegraphics[width = \textwidth]{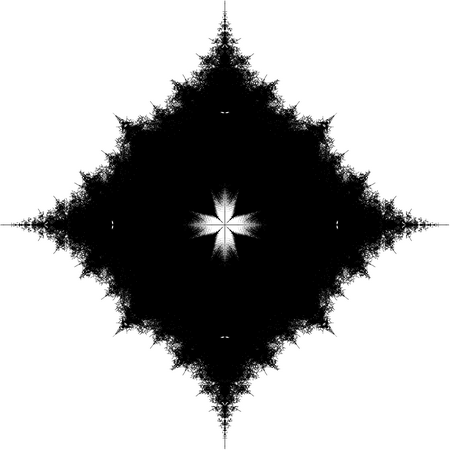}
\caption{\footnotesize spectra of all $20$-periodic tridiagonal sign operators}
\end{subfigure}

\caption{}
\end{center}
\end{figure}

\section{Three Lemmas and a Theorem}

As mentioned in the introduction it was shown in \cite{ChaChoLi2} that $\sigma_\infty$ is a subset of $\pi_\infty$. The following theorem confirms the conjecture by Chandler-Wilde, Chonchaiya and Lindner (\cite{ChaChoLi},\cite{ChaChoLi2}) that $\sigma_{\infty}$ is even dense in $\pi_{\infty}$.

\begin{thm} \label{theorem}
$\sigma_{\infty}$ is a dense subset of $\pi_{\infty}$.
\end{thm}

\setcounter{thm}{0}

First observe that it is enough to consider pairs $(k,l)$ with $l = (1, \ldots, 1)$ since every $A^{k,l}_{fin}$ is unitarily equivalent, via a diagonal so-called gauge transform (\cite{TrEm}), to some $A^{\tilde{k},\tilde{l}}_{fin}$ with $\tilde{k} \in \set{\pm 1}^n$ and $\tilde{l} = (1, \ldots, 1)$. Similarly, every $A^{k,l}_{per}$ is unitarily equivalent to some $A^{\tilde{k},\tilde{l}}_{per}$ with $\tilde{k} \in \set{\pm 1}^m$ and $\tilde{l} = (1, \ldots, 1)$. Thus we omit the index $l$ from now on.

In the proof of Theorem \ref{theorem} we will use three auxiliary lemmas. The first one is well-known in the theory of block Laurent operators. It generalizes the symbol calculus for Laurent operators to periodic operators. Just observe that an $m$-periodic operator can be interpreted as a block Laurent operator with block size $m \times m$. Using a block Fourier transform it is then not hard to see that the operator $A^{k}_{per}$ is unitarily equivalent to the (generalized) multiplication operator $M_{a^k} \from L^2([0,2\pi),\C^m) \to L^2([0,2\pi),\C^m)$ defined by the matrix valued function
\[a^k \from [0,2\pi) \to \C^{m \times m}, \quad \phi \mapsto a^k(\phi) := \begin{pmatrix} 0 & 1 & & k_me^{i\phi} \\ k_1 & \ddots & \ddots & \\ & \ddots & \ddots & 1 \\ e^{-i\phi} & & k_{m-1} & 0 \end{pmatrix}.\]
The first lemma is then an easy consequence. For a proof and more material on this subject see e.g. \cite[Theorem 2.93]{BoeSi}, \cite[Theorem 4.4.9]{Davies} or \cite[Chap. VIII, Theorem 5.1]{GoFe}.
\begin{lem} \label{lemma1}
Let $m \in \N$ and $k \in \set{\pm 1}^m$. Then we have 
\[\spec(A^k_{per}) = \set{\lambda \in \C : \det(a^k(\phi) - \lambda I_m) = 0 \text{ for some } \phi \in [0,2\pi)}.\]
\end{lem}

In the case of tridiagonal periodic operators, only the constant term of $\det(a^k(\phi) - \lambda I_m)$ (as a polynomial of $\lambda$) depends on $\phi$ as the next lemma shows. This leads to the fact that the spectrum of every periodic operator $A^k_{per}$ can be written as $p^{-1}([-2,2])$ for some polynomial $p$ (cf. Figure 2). For an explicit formula of $p$ we refer to \cite{Ha2}.

\begin{figure}[ht!]
\begin{center}
\includegraphics[width = \textwidth, trim = 15mm 5mm 10mm 5mm, clip = true]{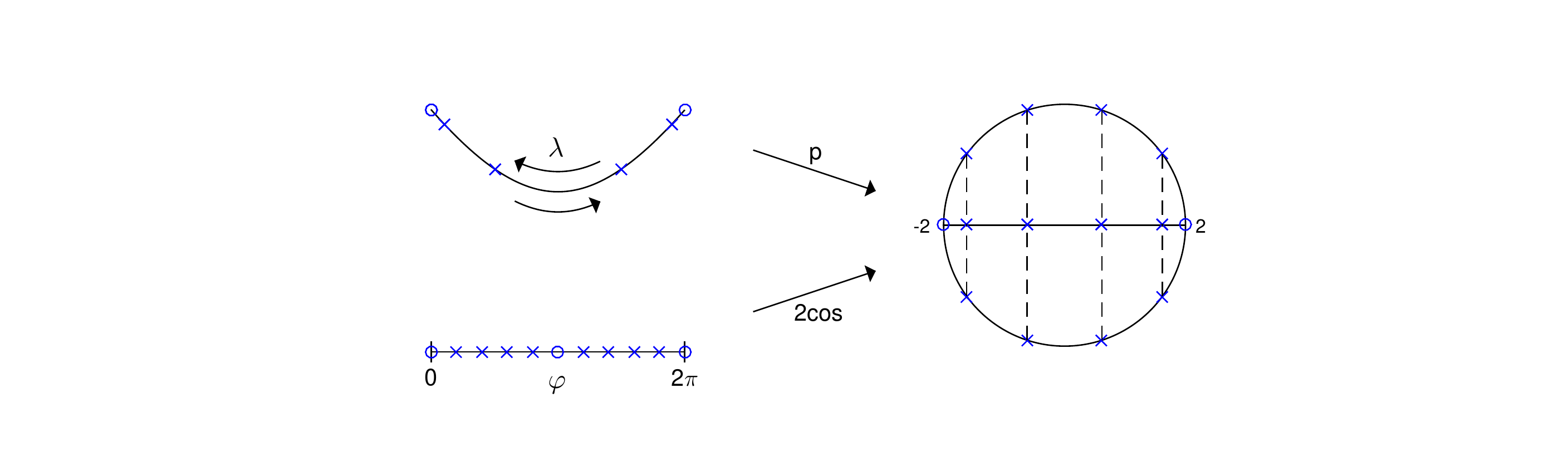}
\end{center}

\caption{a sketch of the maps involved in the computation of the spectrum of a periodic operator}
\end{figure}

\begin{lem} \label{lemma2}
Let $m \in \N$ and $k \in \set{\pm 1}^m$. Then there exists a polynomial $p \from \C \to \C$ of degree $m$ (depending on $k$) such that
\[\det(a^k(\phi) - \lambda I_m) = (-1)^m\left(p(\lambda) -e^{i\phi}\prod\limits_{j = 1}^m k_j - e^{-i\phi}\right)\]
for all $\phi \in [0,2\pi)$.
\end{lem}

\begin{proof}
Using the Leibniz formula
\[\det(a^k(\phi) - \lambda I_m) = \sum\limits_{\tau \in S_m} \text{sign}(\tau) \prod\limits_{i = 1}^n a_{i,\tau_i}^k(\phi),\]
it is easily seen that the only surviving term containing $e^{i\phi}$ but not $e^{-i\phi}$ is $(-1)^{m+1}e^{i\phi}\prod\limits_{j = 1}^m k_j$ and similar for $e^{-i\phi}$. Thus the assertion follows. Alternatively one could also apply Laplace's formula twice to get the same result (see \cite{Ha2} for details).
\end{proof}

The third lemma is a discrete version of Lemma \ref{lemma1}. Although the result is well-known (see e.g. \cite[Section 2.1]{BoeGr}), we include a short proof for the reader's convenience.

\begin{lem} \label{lemma3}
Let $n,m \in \N$ and $A,B,C \in \C^{m \times m}$. Furthermore, denote by $e^{i\xi_1}, \ldots, e^{i\xi_n}$ the $n$-th roots of unity, i.e. $\xi_j := \frac{2j}{n}\pi$ for $j \in \set{1, \ldots, n}$. Then the following block matrices are unitarily equivalent:
\begin{align*}
&\qquad \qquad \quad \; T_1 := \begin{pmatrix} B & C & & A \\ A & \ddots & \ddots & \\ & \ddots & \ddots & C \\ C & & A & B \end{pmatrix} \in \C^{nm \times nm},\\
T_2 &:= \begin{pmatrix} Ae^{i\xi_1} + B + Ce^{-i\xi_1} & & & \\ & \ddots & & \\ & & \ddots & \\ & & & Ae^{i\xi_n} + B + Ce^{-i\xi_n} \end{pmatrix} \in \C^{nm \times nm}\\
& \; = \diag(Ae^{i\xi_1} + B + Ce^{-i\xi_1}, \ldots, Ae^{i\xi_n} + B + Ce^{-i\xi_n}).
\end{align*}
\end{lem}

\begin{proof}
With the help of the Kronecker product $\otimes$, we can write $T_1 = P \otimes A + I_n \otimes B + P^* \otimes C$, where
\[P := \begin{pmatrix} 0 & & & 1 \\ 1 & \ddots & & \\ & \ddots & \ddots & \\ & & 1 & 0 \end{pmatrix} \in \C^{n \times n}.\]
$P$ is unitarily equivalent to the diagonal matrix
\[D := \begin{pmatrix} e^{i\xi_1} & & \\ & \ddots & \\ & & e^{i\xi_n} \end{pmatrix} \in \C^{n \times n}.\]
Thus $T_1$ is unitarily equivalent to $D \otimes A + I_n \otimes B + D^* \otimes C$, which is exactly $T_2$.
\end{proof}

\begin{proof}
(Theorem)\\
That $\sigma_\infty \subset \pi_\infty$ holds was proven in \cite[Theorem 4.1]{ChaChoLi2}. So let $m \in \N$, $k \in \set{\pm 1}^m$ and consider the operator $A^k_{per}$. W.l.o.g. we can assume that the number of $-1$'s in $k$ is even because we can always double the period without changing the operator. This implies that we have
\begin{align} \label{eq1}
\det(a^k(\phi) - \lambda I_m) = (-1)^m(p(\lambda) - 2\cos(\phi))
\end{align}
in Lemma \ref{lemma2}. Using Lemma \ref{lemma1}, we get $\spec(A^k_{per}) = p^{-1}([-2,2])$. This implies, in particular, that the set
\[S^k_\infty := \set{\lambda \in \C : \det(a^k(\phi) - \lambda I_m) = 0 \text{ for some } \phi \in \pi(\Q \setminus \Z)} = p^{-1}(2\cos(\pi(\Q \setminus \Z)))\]
is dense in $\spec(A^k_{per})$, cf. again Figure 2. Let $n \in \N$ and $\xi_j := \frac{2j}{n}\pi$ for $j \in \set{1, \ldots, n}$ as above. We will show that for every $n \in \N$ there exists a finite matrix $A^l_{fin} \in \C^{(nm-1) \times (nm-1)}$ such that
\begin{equation} \label{eq2}
S^k_n := \bigcup\limits_{j \in \set{1, \ldots, n-1} \setminus \{\frac{n}{2}\}} \spec(a^k(\xi_j)) \subset \spec(A^l_{fin}) \subset \sigma_\infty.
\end{equation}
Since
\[S^k_\infty = \bigcup\limits_{n \in \N} S^k_n\]
holds, this implies
\[S^k_\infty = \bigcup\limits_{n \in \N}\bigcup\limits_{j \in \set{1, \ldots, n-1} \setminus \{\frac{n}{2}\}} \spec(a^k(\xi_j)) \subset \sigma_\infty.\]
Because $k \in \set{\pm 1}^m$ and $m \in \N$ were arbitrary, this implies that $\sigma_\infty$ contains a dense subset of $\pi_\infty$. Hence the assertion follows.

In order to construct this matrix $A^l_{fin}$, we arrange the matrices $a^k(\xi_j)$ in a block diagonal matrix and use Lemma \ref{lemma3}:
\[\begin{pmatrix} a^k(\xi_1) & & & \\ & a^k(\xi_2) & & \\ & & \ddots & \\ & & & a^k(\xi_n) \end{pmatrix} \cong \begin{pmatrix} B & C & & A \\ A & \ddots & \ddots & \\ & \ddots & \ddots & C \\ C & & A & B \end{pmatrix} =: M \in \C^{nm \times nm},\]
where
\[A = \begin{pmatrix} & & & k_m \\ & & & \\ & & & \\ & & & \end{pmatrix}, B = \begin{pmatrix} 0 & 1 & & \\ k_1 & \ddots & \ddots & \\ & \ddots & \ddots & 1 \\ & & k_{m-1} & 0 \end{pmatrix}, C = \begin{pmatrix} & & & \\ & & & \\ & & & \\ 1 & & & \end{pmatrix}.\]
Now note that by \eqref{eq1}, $a^k(\xi_j)$ and $a^k(\xi_{n-j})$ have the same spectrum for all $j \in \set{1, \ldots, n-1}$. This implies that every eigenspace of $M$ corresponding to some $\lambda \in S^k_n$ is at least two-dimensional. So we can choose two linearly independent eigenvectors $v,w \in \C^{nm}$ of $M$ corresponding to $\lambda$ and a non-trivial linear combination $x := \alpha v + \beta w \in \C^{nm}$, $\alpha,\beta \in \C$ such that $x_1 = 0$. Of course, $x$ is still an eigenvector of $M$ corresponding to $\lambda$. Let $A^l_{fin} := (M_{ij})_{2 \leq i,j, \leq nm}$. Then we have
\[
\left(\begin{array}{cc}
0 &
\begin{array}{ccccc}
1 & & & & k_m
\end{array} \\
\cline{2-2}
\begin{array}{c} k_1 \\ \\ \\ \\ 1 \end{array} &
\multicolumn{1}{|c|}{A^l_{fin}} \\
\cline{2-2}
\end{array}
\right)
\left(\begin{array}{c}
0 \\ \hline
\multicolumn{1}{|c|}{x_2}\\
\multicolumn{1}{|c|}{}\\
\multicolumn{1}{|c|}{\vdots}\\
\multicolumn{1}{|c|}{}\\
\multicolumn{1}{|c|}{x_{nm}}\\
\hline
\end{array}
\right)
=
\lambda
\left(\begin{array}{c}
0 \\ \hline
\multicolumn{1}{|c|}{x_2}\\
\multicolumn{1}{|c|}{}\\
\multicolumn{1}{|c|}{\vdots}\\
\multicolumn{1}{|c|}{}\\
\multicolumn{1}{|c|}{x_{nm}}\\
\hline
\end{array}
\right).
\]
This implies that $\lambda$ is also an eigenvalue of $A^l_{fin}$, hence \eqref{eq2} follows.
\end{proof}

Our result heavily depends on the fact that all elements on the first sub- and superdiagonal have the same absolute value (and on tridiagonality of course). These cases are usually called critical or degenerated. Here degenerated refers to the fact that if all elements on the first sub- and superdiagonal have the same absolute value, the spectra of the periodic operators, which are loops in general, are degenerated to an arc. Changing the alphabet in the first subdiagonal to $\set{\pm \sigma}$, where $\sigma \in (0,1)$, as considered in \cite{ChaDa} or \cite{Ha} for example, the conclusion fails. This was of course to be expected because a finite matrix
\[\begin{pmatrix} 0 & l_1 & & \\ k_1 & \ddots & \ddots & \\ & \ddots & \ddots & l_n \\ & & k_n & 0 \end{pmatrix} \in \C^{(n+1) \times (n+1)}, k \in \set{\pm \sigma}^n, l \in \set{\pm 1}^n\]
is always similar (diagonal gauge transform, see \cite{TrEm}) to a matrix
\[\begin{pmatrix} 0 & \tilde{l}_1 & & \\ \tilde{k}_1 & \ddots & \ddots & \\ & \ddots & \ddots & \tilde{l}_n \\ & & \tilde{k}_n & 0 \end{pmatrix} \in \C^{(n+1) \times (n+1)}, \tilde{k} \in \set{\pm \sqrt{\sigma}}^n, \tilde{l} \in \set{\pm \sqrt{\sigma}}^n.\]
This is no longer true for periodic operators. For tridiagonal operators on $\ell^2(\Z)$ we can only shift phases to the other side. This remaining freedom, however, can be used to prove Theorem \ref{theorem} for arbitrary alphabets as long as all elements share the same absolute value. This only needs a small modification of Lemma \ref{lemma3} and a refinement of \cite[Theorem 4.1]{ChaChoLi2}. The details are left to the reader.

\bigskip

\noindent
{\bf Author's address:}

\medskip

\noindent
Raffael Hagger \hfill{\tt raffael.hagger@tuhh.de}\\
Institute of Mathematics\\
Hamburg University of Technology\\
Schwarzenbergstr. 95 E\\
D-21073 Hamburg\\
GERMANY

\end{document}